\documentclass[11pt]{article}
\usepackage{amsmath,amssymb,amsthm,geometry,hyperref,booktabs}
\newtheorem{definition}{Definition}[section]
\newtheorem{theorem}{Theorem}[section]
\newtheorem{lemma}[theorem]{Lemma}
\newtheorem{corollary}[theorem]{Corollary}
\newtheorem{proposition}[theorem]{Proposition}
\newtheorem{remark}[theorem]{Remark}
\newtheorem{example}[theorem]{Example}

\title{The Domination Reciprocal Mean Square Index}
\author{
  Meysam Taheri-Dehkordi \\[1mm]
  \small University of Applied Science and Technology, Tehran, Iran \\
  \small \texttt{m.taheri@uast.ac.ir}
}
\date{}

\begin{document}
\maketitle

\begin{abstract}
We introduce the Domination Reciprocal Mean Square index \(DRMS\), a
new domination-degree-based topological index obtained by replacing the
ordinary degree in the reciprocal mean square index with the domination
degree. We establish elementary bounds in terms of graph size and extremal
domination degrees, derive a one-parameter family of sharp lower bounds via
the power mean, and obtain sharp bounds in terms of several well-known
domination-degree-based indices, with a full characterization of the equality
cases. We also compute closed-form expressions for \(DRMS\) on several
standard graph families and for the corona product of two graphs.
\end{abstract}

\section{Introduction and Motivation}
Throughout, \( G \) denotes a simple connected graph with vertex set \( V(G) \) and edge set \( E(G) \). Its order and size are \( n = |V(G)| \) and \( m = |E(G)| \), respectively. By the degree of a vertex \( u \), denoted \( d(u) \), we mean the number of its neighbors in \( G \). Any graph-theoretic terminology or notation not defined in this work may be found in \cite{bondy, char, diestel}.\\
Numerical graph invariants known as topological indices are of key importance in mathematical chemistry. Their indispensability in chemical graph theory has been established through their widespread adoption in QSAR/QSPR modeling. When the definition of such an index involves only vertex degrees, it is classified as a vertex-degree-based index \cite{Gutman2013, furtula2013, cao2014}. \\
 For a graph \( G \), a subset \( D \subseteq V(G) \) is called a dominating set if each vertex of \( G \) lies in \( D \) or has a neighbor in \( D \). The domination degree of a vertex $u \in V(G)$ is the number of minimal dominating sets of $G$ containing $u$. This quantity is denoted by $d_d(u)$. The symbols $\Delta_d(G)=\Delta_d$ and $\delta_d(G)=\delta_d$ denote the largest and smallest domination degrees among the vertices of a graph $G$, respectively. On the other hand, 
\[
\delta_d=\min_{u\in V(G)}d_d(u),\qquad
\Delta_d=\max_{u\in V(G)}d_d(u).\qquad
\]
The graph \( G \) is \( k \)-domination regular whenever \( d_d(u) = k \) for all \( u \in V(G) \).
Recently, several topological indices based on domination degrees have been
introduced. In particular, Hanan Ahmed et al.~\cite{Hanan} introduced the first
family of such indices, which has since been extended in various directions. Domination topological indices are chiefly intended to apply graph domination concepts in order to obtain numerical values for chemical graphs. References \cite{Ahmed1, barman2026, HananD, DSOM} offer further details and applications concerning these indices for different families of standard graphs and graph operations.
The Reciprocal Mean Square $(RMS)$ index of a graph is defined as
\[
RMS(G)=\sum_{uv\in E(G)}\frac{1}{d(u)^2+d(v)^2},
\]
The reciprocal mean square index was formally proposed and subjected to a systematic investigation in \cite{AliRMS1, AliRMS2}.
Replacing the ordinary degree $d(u)$ by the domination degree $d_d(u)$,
we obtain the Domination Reciprocal Mean Square index $(DRMS)$.\\
Accordingly, in this paper we introduce a new index called the Domination Reciprocal Mean Square index (\(DRMS\)), which is obtained by replacing the ordinary degree with the domination degree in the \(RMS\) index. The main purpose of this work is to study this index systematically from both theoretical and computational perspectives. First, we establish elementary bounds in terms of the graph size and the extremal domination degrees (\(\delta_d\) and \(\Delta_d\)), and characterize their equality cases. Then, we derive a one-parameter family of sharp lower bounds via the power mean \(M_p(x)\), which unifies several classical bounds. Next, we obtain sharp bounds in terms of well-known domination-degree-based indices, including the modified forgotten domination index, the second domination Zagreb index, the first domination Banhatti index, the domination Sombor index, and the domination elliptic Sombor index, with a full characterization of the equality cases. Moreover, we compute closed-form expressions for \(DRMS\) on several standard graph families as well as for the corona product of two graphs. 

\section{Preliminaries}
This section collects the main definitions and auxiliary results used
throughout the paper. We first introduce the Domination Reciprocal Mean Square
index \(DRMS\), together with the domination-degree-based indices that will
appear in the subsequent bounds. We then recall the notions of a (minimal)
dominating set and of a private neighbour, and prove two elementary lemmas
that will be used repeatedly.

\begin{definition}
For a connected graph $G$,
\[
DRMS(G)=\sum_{uv\in E(G)}\frac{1}{d_d^2(u)+d_d^2(v)}.
\]
\end{definition}
Throughout, we use the following notations for every edge $e=uv$,
\[
a_e=d_d(u),\quad b_e=d_d(v),\quad
x_e=d_d^2(u)+d_d^2(v),\quad
\sigma_e=d_d(u)+d_d(v).
\]
The following domination-degree-based indices will be used repeatedly.

\begin{definition}[Modified Forgotten Domination Index \cite{HananD}]
\[
DF^*(G)=\sum_{uv\in E(G)}\bigl(d_d^2(u)+d_d^2(v)\bigr)=\sum_{e\in E(G)}x_e.
\]
\end{definition}

\begin{definition}[Second Domination Zagreb Index \cite{Hanan}]
\[
DM_2(G)=\sum_{uv\in E(G)}d_d(u)\,d_d(v).
\]
\end{definition}

\begin{definition}[First Domination Banhatti Index \cite{kulli2023}] \label{2.4}
\[
DB_1(G)=\sum_{uv\in E(G)}\frac{d_d(u)+d_d(v)}{2}=\sum_{e\in E(G)}\frac{\sigma_e}{2}.
\]
\end{definition}

\begin{definition}[Hyper Domination Index \cite{HananD}]
\label{def:DH}
\[
DH(G)=\sum_{uv\in E(G)}\bigl(d_d(u)+d_d(v)\bigr)^2=\sum_{uv\in E(G)}\bigl(\sigma_e\bigr)^2.
\]
\end{definition}

\begin{definition}[Domination Elliptic Sombor Index \cite{BD26}]
\label{def:DHar}
\[
DESO(G)=\sum_{uv\in E(G)}(d_d(u)+d_d(v))\sqrt{d_d^2(u)+d_d^2(v)}=\sum_e\sigma_e\sqrt{x_e}.
\]
\end{definition}

\begin{definition}[Domination Sombor Index \cite{DSOM}]
\[
SO_d(G)=\sum_{uv\in E(G)}\sqrt{d_d^2(u)+d_d^2(v)}=\sum_{e\in E(G)}\sqrt{x_e}.
\]
\end{definition}

\begin{definition}[Modified Domination Sombor Index \cite{kulli2023-2}]
\[
{}^{m}SO_d(G)=\sum_{uv\in E(G)}\frac{1}{\sqrt{d_d^2(u)+d_d^2(v)}}=\sum_{e\in E(G)}x_e^{-1/2}.
\]
\end{definition}

For a vertex $v\in V(G)$, the \emph{closed neighbourhood} of $v$ is
$N[v]=\{v\}\cup\{w\in V(G): vw\in E(G)\}$, and for $S\subseteq V(G)$ we put
$N[S]=\bigcup_{u\in S}N[u]$.

\begin{definition}
A set $S\subseteq V(G)$ is a dominating set of $G$ if $N[S]=V(G)$,
i.e.\ every vertex of $G$ lies in $S$ or is adjacent to a vertex of $S$.
A dominating set $S$ is minimal if no proper subset of $S$ is dominating;
equivalently, $S\setminus\{u\}$ is not dominating for every $u\in S$.
\end{definition}

\begin{definition}
Let $S\subseteq V(G)$ and $u\in S$. A vertex $w\in N[u]$ is called a
private neighbour of $u$ with respect to $S$ if $N[w]\cap S=\{u\}$,
i.e.\ $u$ is the only vertex of $S$ that dominates $w$.
(Note that $w=u$ is allowed.)
\end{definition}

\begin{lemma}
\label{lem:private}
Let $S$ be a dominating set of $G$. Then $S$ is minimal if and only if every
$u\in S$ has a private neighbour.
\end{lemma}
\begin{proof}
Let $S$ be a dominating set of $G$. We prove both directions.\\
 Assume $S$ is minimal, and let $u \in S$ be arbitrary. By minimality, $S \setminus \{u\}$ is not dominating. Hence there exists $w \in V(G)$ such that
$
N[w] \cap (S \setminus \{u\}) = \varnothing.
$
Since $S$ is dominating, $N[w] \cap S \neq \varnothing$. Because $S = (S \setminus \{u\}) \cup \{u\}$, we obtain
$
N[w] \cap S = \{u\}.
$\\
In particular $u \in N[w]$, and by symmetry of the closed neighbourhood, $w \in N[u]$. Thus $w$ is a private neighbour of $u$ with respect to $S$. Since $u \in S$ was arbitrary, every vertex of $S$ has a private neighbour.\\
Now, assume every $u \in S$ has a private neighbour. Fix $u \in S$ and let $w \in N[u]$ be a private neighbour of $u$, i.e.
$
N[w] \cap S = \{u\}.
$
Then
$
N[w] \cap (S \setminus \{u\}) = \varnothing,
$
and $w \notin S \setminus \{u\}$ (otherwise $w \in N[w] \cap (S \setminus \{u\})$, a contradiction). Hence $w$ is not dominated by $S \setminus \{u\}$, so $S \setminus \{u\}$ is not dominating. Since $u \in S$ was arbitrary, no proper subset of $S$ obtained by deleting a single vertex is dominating; equivalently, $S$ is minimal.
\end{proof}

\begin{lemma}\label{lem:delta1}
Every vertex of a graph lies in a minimal dominating set; hence
$\delta_d\ge1$.
\end{lemma}
\begin{proof}
Let $v \in V(G)$ be arbitrary.\\
Step 1. Extend the independent set $\{v\}$ to a maximal
independent set $I$ (possible by a greedy argument). Then $v \in I$.\\
Step 2. We claim that $I$ is a dominating set. Suppose
otherwise: there exists $x \in V(G)$ such that $x \notin I$ and $x$ is
adjacent to no vertex of $I$. Then $I \cup \{x\}$ is also an independent
set, contradicting the maximality of $I$. Hence $I$ is dominating.\\
Step 3. We claim that $I$ is a \emph{minimal}
dominating set. Let $u \in I$. Since $I$ is independent, no vertex of
$I \setminus \{u\}$ is equal or adjacent to $u$; in particular,
$u \notin I \setminus \{u\}$ and $N[u] \cap (I \setminus \{u\}) = \varnothing$.
Hence $I \setminus \{u\}$ does not dominate $u$. Since $u \in I$ was
arbitrary, $I$ is minimal.\\
By Steps 1--3, $I$ is a minimal dominating
set containing $v$. Since $v \in V(G)$ was arbitrary, every vertex of
$G$ lies in a minimal dominating set, which implies
\[
\delta_d = \min_{v \in V(G)} d_d(v) \ge 1. \qedhere
\]
\end{proof}

\section{Elementary Bounds and Power Means}
We begin with two elementary bounds for \(DRMS(G)\) in terms of the graph size
\(m\) and the extremal domination degrees \(\delta_d\) and \(\Delta_d\), and we
characterize the graphs for which these bounds are attained. We then introduce
the power mean \(M_p(x)\) of the numbers \(x_e = d_d^2(u) + d_d^2(v)\) and
derive a one-parameter family of lower bounds that unifies several of the
bounds appearing later in the paper.
\begin{theorem}
\label{thm:elementary}
For every connected graph $G$,
\[
\frac{m}{2\Delta_d^2} \le DRMS(G) \le \frac{m}{2\delta_d^2}.
\]
Equality holds in either bound if and only if $G$ is
$k$-domination-regular for some $k \ge 1$; in that case,
$
DRMS(G) = \frac{m}{2k^2}.
$
\end{theorem}
\begin{proof}
For every edge $uv \in E(G)$, since $\delta_d \le d_d(x) \le \Delta_d$ for all $x$,
\begin{equation}\label{eq:basic-ineq}
2\delta_d^2 \le d_d^2(u) + d_d^2(v) \le 2\Delta_d^2.
\end{equation}
\emph{Lower bound.} From the right side of \eqref{eq:basic-ineq},
$d_d^2(u)+d_d^2(v) \le 2\Delta_d^2$ for every edge, so
$\frac{1}{d_d^2(u)+d_d^2(v)} \ge \frac{1}{2\Delta_d^2}$.
Summing over the $m$ edges gives $DRMS(G) \ge \frac{m}{2\Delta_d^2}$.\\
\emph{Upper bound.} Similarly, from the left side of \eqref{eq:basic-ineq},
$\frac{1}{d_d^2(u)+d_d^2(v)} \le \frac{1}{2\delta_d^2}$, and summing gives
$DRMS(G) \le \frac{m}{2\delta_d^2}$.\\
\emph{Equality.} Each summand in $DRMS(G) - \frac{m}{2\Delta_d^2}$ is nonnegative,
so the sum vanishes if and only if every summand vanishes, i.e.\ if and only if
$d_d^2(u)+d_d^2(v) = 2\Delta_d^2$ for every edge $uv$. Since
$d_d^2(x) \le \Delta_d^2$, this forces $d_d(u)=d_d(v)=\Delta_d$ on every edge,
hence $d_d \equiv \Delta_d$ by connectedness. The upper-bound case is symmetric:
equality holds if and only if $d_d \equiv \delta_d$. In either case $d_d\equiv k$, and
$
DRMS(G) = \sum_{uv \in E} \frac{1}{2k^2} = \frac{m}{2k^2}. \qedhere
$
\end{proof}
\begin{corollary}
If $G$ is $k$-domination-regular with $m$ edges, then
\[
DRMS(G)=\frac{m}{2k^2}.
\]
\end{corollary}
\begin{theorem}\label{thm:max}
For every connected graph $G$ of order $n\ge2$,
\[
DRMS(G)\ \le\ \frac m2\ \le\ \frac{n(n-1)}4 .
\]
Equality holds in the first inequality if and only if $d_d\equiv1$, and in the
second if and only if $G=K_n$. In particular,
$\max_G DRMS(G)=n(n-1)/4$ over connected graphs of order $n$, and it is
attained only by $K_n$. Moreover, every tree $T$ of order $n$ satisfies
$DRMS(T)\le (n-1)/2$.
\end{theorem}

\begin{proof}
By Lemma~\ref{lem:delta1} we have $d_d\ge1$, so $x_e=a_e^{2}+b_e^{2}\ge2$ and
$1/x_e\le1/2$ for every edge $e$, with equality if and only if
$a_e=b_e=1$. Summing over $E(G)$ gives $DRMS(G)\le m/2$, with equality if and
only if $a_e=b_e=1$ for every edge $e$. Since $G$ is connected and $n\ge2$,
every vertex lies on an edge, so this holds if and only if $d_d\equiv1$ (and
if $d_d\equiv1$, then $x_e=2$ for all $e$, so equality indeed holds). Next,
$m\le n(n-1)/2$ with equality if and only if $G=K_n$, which gives the second
inequality and its equality case. Every vertex of $K_n$ lies in exactly one
minimal dominating set (the singletons), so $d_d\equiv1$ on $K_n$ and
$DRMS(K_n)=\frac12\cdot\frac{n(n-1)}2=\frac{n(n-1)}4$. Hence the maximum is
$n(n-1)/4$. If $DRMS(G)=n(n-1)/4$, then both inequalities of the chain are
equalities; in particular $m=n(n-1)/2$, so $G=K_n$. Finally, a tree has
$m=n-1$ edges, so $DRMS(T)\le m/2=(n-1)/2$.
\end{proof}

For a real parameter \(p \neq 0\), the \emph{power mean of order \(p\)} of the
numbers \(x_e = d_d^2(u) + d_d^2(v)\) is defined by
\[
M_p(x) = \left(\frac{1}{m}\sum_{e \in E(G)} x_e^{\,p}\right)^{1/p}.
\]
At \(p = 0\), we set
\[
M_0(x) = \left(\prod_{e \in E(G)} x_e\right)^{1/m},
\]
which is the limit of \(M_p(x)\) as \(p \to 0\).\\
In the next theorem, we construct a lower bound for \(DRMS(G)\) using the
power mean \(M_p(x)\) of the numbers \(x_e = d_d^2(u) + d_d^2(v)\).

\begin{theorem}\label{thm:pmean}
For every connected graph \(G\) and every \(p \in [-1, 1]\),
\[
DRMS(G) \ge \frac{m}{M_p(x)}.
\]
For \(p > -1\), equality holds if and only if \(x_e\) is constant over
\(E(G)\); for \(p = -1\), equality always holds.
\end{theorem}
\begin{proof}
Since
\[
DRMS(G) = \sum_{e \in E(G)} x_e^{-1}
= \left( \frac{1}{m} \sum_{e \in E(G)} x_e^{-1} \right)^{-1} \cdot \frac{1}{m} \cdot m
= \frac{m}{M_{-1}(x)},
\]
the claimed inequality \(DRMS(G) \ge m/M_p(x)\) is equivalent to
\[
M_{-1}(x) \le M_p(x).
\]
This is the monotonicity of the power mean in its order: for \(q \le p\),
\(M_q(x) \le M_p(x)\). Taking \(q = -1\) and \(p \ge -1\) gives the claim.\\
\emph{Equality.} The function \(p \mapsto M_p(x)\) is strictly increasing
unless all \(x_e\) coincide; hence for \(p > -1\), equality \(M_{-1}(x) = M_p(x)\)
holds if and only if \(x_e\) is constant over \(E(G)\). For \(p = -1\), both
sides equal \(DRMS(G)\) trivially, so equality always holds.
\end{proof}

\begin{corollary}\label{cor:pmean-special}
Under the notation of Theorem~\ref{thm:pmean}, the bound takes the following
forms for special values of \(p\):
\begin{align*}
p = -\tfrac{2}{3}:&\quad DRMS(G) \ge m^{-1/2} \Bigl(\sum_{e} x_e^{-2/3}\Bigr)^{3/2}, \\
p = -\tfrac{1}{2}:&\quad DRMS(G) \ge \frac{{}^{m}SO_d(G)^2}{m}, \\
p = 0:&\quad DRMS(G) \ge m \Bigl(\prod_{e} x_e\Bigr)^{-1/m}, \\
p = \tfrac{1}{2}:&\quad DRMS(G) \ge \frac{m^3}{SO_d(G)^2}, \\
p = 1:&\quad DRMS(G) \ge \frac{m^2}{DF^*(G)}.
\end{align*}
\end{corollary}

\section{Sharp Bounds via Domination Degree-Based Indices}
In this section, we derive a collection of sharp bounds for \(DRMS(G)\) in
terms of the domination-degree-based indices introduced in Section~2. Each
bound is obtained by applying a suitable classical inequality to the sequence
\((x_e)_{e \in E(G)}\), and we characterize the graphs for which equality is
attained.

In the following three theorems, we establish bounds for $DRMS(G)$ in terms
of the modified forgotten domination index $DF^*(G)$.
\begin{theorem}
\label{thm:cauchy}
For every connected graph $G$,
\[
DRMS(G) \ge \frac{m^2}{DF^*(G)}.
\]
Equality holds if and only if $x_e$ is constant over all edges $e \in E(G)$.
\end{theorem}

\begin{proof}
Apply Cauchy--Schwarz to the sequences
\[
p_e := \sqrt{x_e}, \qquad q_e := \frac{1}{\sqrt{x_e}} \qquad (e \in E(G)).
\]
Then $p_e q_e = 1$, $p_e^2 = x_e$, and $q_e^2 = 1/x_e$, so
\[
\Bigl(\sum_{e \in E} 1\Bigr)^2
\le \Bigl(\sum_{e \in E} x_e\Bigr)\Bigl(\sum_{e \in E} \frac{1}{x_e}\Bigr)
= DF^*(G) \cdot DRMS(G).
\]
Since $x_e \ge 2\delta_d^2 \ge 2 > 0$ for every edge (by Lemma \ref{lem:delta1}), we have
$DF^*(G) > 0$, so dividing by $DF^*(G)$ yields
\[
DRMS(G) \ge \frac{m^2}{DF^*(G)}.
\]
Equality in Cauchy--Schwarz holds if and only if the sequences $p_e$ and
$q_e$ are proportional, which means there exists a constant $\lambda$ such
that $p_e = \lambda\, q_e$ for every $e \in E(G)$. Substituting
$p_e = \sqrt{x_e}$ and $q_e = 1/\sqrt{x_e}$ gives
$
\sqrt{x_e} = \frac{\lambda}{\sqrt{x_e}},
$
and hence $x_e = \lambda$ for every edge. Therefore equality holds if and only
if $x_e$ is constant over all edges $e \in E(G)$.
\end{proof}

\begin{theorem}
\label{thm:interval}
For every connected graph $G$,
\[
DRMS(G) \le \frac{2m(\delta_d^2 + \Delta_d^2) - DF^*(G)}{4\delta_d^2 \Delta_d^2}.
\]
Equality holds if and only if $x_e \in \{2\delta_d^2, 2\Delta_d^2\}$ for every
edge $e \in E(G)$.
\end{theorem}

\begin{proof}
Set $x_- = 2\delta_d^2$ and $x_+ = 2\Delta_d^2$. For every $x \in [x_-, x_+]$,
\begin{equation}\label{eq:interval-aux}
\frac{1}{x} \le \frac{x_- + x_+ - x}{x_-x_+}.
\end{equation}
Indeed, multiplying both sides by $x_-x_+x > 0$ gives
$x_-x_+ \le (x_- + x_+ - x)x$, which rearranges to $(x - x_-)(x_+ - x) \ge 0$;
the latter holds because $x \in [x_-, x_+]$.\\
By the basic inequality $2\delta_d^2 \le x_e \le 2\Delta_d^2$, every edge
satisfies $x_e \in [x_-, x_+]$. Applying \eqref{eq:interval-aux} to $x = x_e$
and summing over all $e \in E(G)$,
\[
DRMS(G) = \sum_{e \in E} \frac{1}{x_e}
\le \sum_{e \in E} \frac{x_- + x_+ - x_e}{x_-x_+}
= \frac{1}{x_-x_+}\Bigl[(x_- + x_+)m - \sum_{e \in E} x_e\Bigr].
\]
Substituting $x_- = 2\delta_d^2$, $x_+ = 2\Delta_d^2$, $x_-x_+ = 4\delta_d^2 \Delta_d^2$,
and $\sum_{e \in E} x_e = DF^*(G)$ yields the claimed inequality.\\
\emph{Equality condition.} Since the sum
\[
\frac{1}{x_-x_+}\Bigl[(x_- + x_+)m - DF^*(G)\Bigr] - DRMS(G)
= \sum_{e \in E}\Bigl[\frac{x_- + x_+ - x_e}{x_-x_+} - \frac{1}{x_e}\Bigr]
\]
is a sum of nonnegative terms by \eqref{eq:interval-aux}, it vanishes if and
only if every term vanishes, that is, if and only if
\[
\frac{x_- + x_+ - x_e}{x_-x_+} = \frac{1}{x_e} \quad \text{for every } e \in E(G).
\]
By the factorization $(x_e - x_-)(x_+ - x_e) = 0$ established above, this is
equivalent to $x_e = x_-$ or $x_e = x_+$ for every edge. In other words,
equality holds if and only if
\[
x_e \in \{2\delta_d^2,\, 2\Delta_d^2\} \quad \text{for every } e \in E(G).
\qedhere
\]
\end{proof}

\begin{theorem}
\label{thm:kantorovich}
For every connected graph $G$,
\[
DRMS(G) \le \frac{m^2}{DF^*(G)}
\left(\frac{\delta_d^2 + \Delta_d^2}{2\delta_d \Delta_d}\right)^2.
\]
Equality holds if and only if $x_e \in \{2\delta_d^2, 2\Delta_d^2\}$ for every
edge $e \in E(G)$.
\end{theorem}

\begin{proof}
Set $x_- = 2\delta_d^2$ and $x_+ = 2\Delta_d^2$. By the basic inequality,
$x_e \in [x_-, x_+]$ for every edge $e \in E(G)$. The Kantorovich inequality
for a sequence $(x_e)_{e \in E(G)}$ in $[x_-, x_+]$ states that
\begin{equation}\label{eq:kantorovich}
\Bigl(\sum_{e \in E} x_e\Bigr)\Bigl(\sum_{e \in E} \frac{1}{x_e}\Bigr)
\le m^2 \cdot \frac{(x_- + x_+)^2}{4x_-x_+}.
\end{equation}
Substituting $\sum_{e \in E} x_e = DF^*(G)$, $\sum_{e \in E} 1/x_e = DRMS(G)$,
$x_- + x_+ = 2(\delta_d^2 + \Delta_d^2)$, and $x_-x_+ = 4\delta_d^2 \Delta_d^2$ gives
\[
DF^*(G) \cdot DRMS(G) \le m^2 \cdot \frac{(\delta_d^2 + \Delta_d^2)^2}{4\delta_d^2 \Delta_d^2}.
\]
Dividing both sides by $DF^*(G) > 0$ and rewriting the constant factor yields
\[
DRMS(G) \le \frac{m^2}{DF^*(G)}
\left(\frac{\delta_d^2 + \Delta_d^2}{2\delta_d \Delta_d}\right)^2.
\]
\emph{Equality condition.} Equality in \eqref{eq:kantorovich} holds if and only
if each $x_e$ takes only the values $x_-$ and $x_+$, that is, if and only if
$x_e \in \{2\delta_d^2, 2\Delta_d^2\}$ for every edge $e \in E(G)$.
\qedhere
\end{proof}

\begin{theorem}
\label{thm:dm2}
For every connected graph $G$,
\[
DRMS(G)\ \ge\ \frac{\delta_d\Delta_d}{\delta_d^{2}+\Delta_d^{2}}\cdot\frac{m^{2}}{DM_2(G)}.
\]
Equality holds if and only if $\{d_d(u),d_d(v)\}=\{\delta_d,\Delta_d\}$
for every edge $uv\in E(G)$ when $\delta_d<\Delta_d$, and for all edges
(i.e. $G$ is domination-regular) when $\delta_d=\Delta_d$.
\end{theorem}

\begin{proof}
Put $\gamma_d=\dfrac{\delta_d\Delta_d}{\delta_d^{2}+\Delta_d^{2}}$. For an edge $e=uv$
recall $a_e=d_d(u)$ and $b_e=d_d(v)$; by symmetry we may assume $a_e\ge b_e$.
By Lemma \ref{lem:delta1}, $\delta_d\ge 1$, so all quantities below are well defined.

\medskip
\noindent\textit{Step 1 (edgewise inequality).}
Let $r=a_e/b_e$. Since $\delta_d\le b_e\le a_e\le\Delta_d$, we have
$1\le r\le \Delta_d/\delta_d$, with
\begin{equation}\label{eq:dm2-t}
r=\frac{\Delta_d}{\delta_d}
\quad\text{if and only if}\quad
a_e=\Delta_d\ \text{and}\ b_e=\delta_d .
\end{equation}
(Indeed $a_e/b_e\le \Delta_d/\delta_d$ because $a_e\le\Delta_d$ and $b_e\ge\delta_d$,
and equality forces both of these to be equalities.)
The function $f(r)=r+1/r$ satisfies $f'(r)=1-1/r^{2}\ge 0$ on $[1,\infty)$ and is
strictly increasing there. Hence
\[
\frac{a_e^{2}+b_e^{2}}{a_e b_e}=\frac{a_e}{b_e}+\frac{b_e}{a_e}=f(r)
\le f\!\left(\frac{\Delta_d}{\delta_d}\right)
=\frac{\Delta_d}{\delta_d}+\frac{\delta_d}{\Delta_d}
=\frac{\Delta_d^{2}+\delta_d^{2}}{\delta_d\Delta_d},
\]
and, by strict monotonicity and \eqref{eq:dm2-t}, equality holds if and only if
$\{a_e,b_e\}=\{\delta_d,\Delta_d\}$ when $\delta_d<\Delta_d$ (when
$\delta_d=\Delta_d$ both sides equal $2$ for every edge). Taking reciprocals,
\begin{equation}\label{eq:dm2-edgewise}
\frac{1}{x_e}=\frac{1}{a_e^{2}+b_e^{2}}\ \ge\ \gamma_d\,\frac{1}{a_e b_e}\qquad(e\in E(G)),
\end{equation}
with equality if and only if $\{a_e,b_e\}=\{\delta_d,\Delta_d\}$ (or $\delta_d=\Delta_d$).

\medskip
\noindent\textit{Step 2 (summation).}
Summing \eqref{eq:dm2-edgewise} over all edges gives
\begin{equation}\label{eq:dm2-sum}
DRMS(G)\ \ge\ \gamma_d\sum_{e\in E(G)}\frac{1}{a_e b_e}.
\end{equation}

\medskip
\noindent\textit{Step 3 (Cauchy--Schwarz).}
Applying the Cauchy--Schwarz inequality to the sequences
$\sqrt{a_e b_e}$ and $1/\sqrt{a_e b_e}$ $(e\in E(G))$,
\begin{equation}\label{eq:dm2-cs}
m^{2}=\Bigl(\sum_{e\in E(G)}1\Bigr)^{2}
\le \Bigl(\sum_{e\in E(G)}a_e b_e\Bigr)\Bigl(\sum_{e\in E(G)}\frac{1}{a_e b_e}\Bigr)
= DM_2(G)\sum_{e\in E(G)}\frac{1}{a_e b_e}.
\end{equation}
Since $DM_2(G)>0$, this gives $\sum_{e}1/(a_e b_e)\ge m^{2}/DM_2(G)$.
Combining with \eqref{eq:dm2-sum},
\[
DRMS(G)\ \ge\ \gamma_d\,\frac{m^{2}}{DM_2(G)},
\]
which is the claimed inequality.

\medskip
\noindent\textit{Equality.}
Write
\[
DRMS(G)-\gamma_d\frac{m^{2}}{DM_2(G)}
=\sum_{e}\Bigl(\frac1{x_e}-\frac{\gamma_d}{a_e b_e}\Bigr)
+\;\gamma_d\Bigl(\sum_{e}\frac1{a_e b_e}-\frac{m^{2}}{DM_2(G)}\Bigr) .
\]
By \eqref{eq:dm2-edgewise} the first sum is nonnegative, and by
\eqref{eq:dm2-cs} the second sum is nonnegative. Hence the right-hand side is
nonnegative, with equality if and only if both sums vanish.\\
If equality holds, the first bracket is a sum of nonnegative
terms and vanishes, so equality holds in \eqref{eq:dm2-edgewise} for every edge; by Step~1 every
edge $e$ satisfies $\{a_e,b_e\}=\{\delta_d,\Delta_d\}$ (or $\delta_d=\Delta_d$).\\
Suppose every edge satisfies $\{a_e,b_e\}=\{\delta_d,\Delta_d\}$
(or $\delta_d=\Delta_d$). Then equality holds in \eqref{eq:dm2-edgewise} for each $e$, so the first
bracket is $0$. Moreover $a_e b_e=\delta_d\Delta_d$ is constant over $E(G)$, so
the sequences $\sqrt{a_e b_e}$ and $1/\sqrt{a_e b_e}$ are proportional and equality
holds in \eqref{eq:dm2-cs}; the second bracket is $0$ as well.
\end{proof}

\begin{remark}
If $\delta_d<\Delta_d$, a connected graph in which every edge joins a vertex of
degree $\delta_d$ to one of degree $\Delta_d$ is bipartite, with the two classes
given by the two values of $d_d$. For $K_{2,3}$ we have $d_d\in\{4,3\}$,
$m=6$, $DM_2(K_{2,3})=72$, and
$\frac{12}{25}\cdot\frac{36}{72}=0.24=DRMS(K_{2,3})$, so the bound is attained
although $K_{2,3}$ is not domination-regular.
\end{remark}

\begin{theorem}
\label{thm:db1}
For every connected graph $G$,
\[
DRMS(G)\ge
\frac{m^2}{2\,\Delta_d\,DB_1(G)}.
\]
Moreover, equality holds if and only if $G$ is domination-regular.
\end{theorem}

\begin{proof}
For an edge $e=uv$ recall $a_e=d_d(u)$, $b_e=d_d(v)$, $x_e=a_e^{2}+b_e^{2}$ and
$\sigma_e=a_e+b_e$. By Lemma \ref{lem:delta1} we have $\delta_d\ge 1$, so $a_e,b_e,\sigma_e,x_e>0$
and all quotients below are well defined. By Definition \ref{2.4},
$\sum_{e\in E(G)}\sigma_e=2\,DB_1(G)$.\\
Since $a_e,b_e\le\Delta_d$ for every edge $e$,
\[
\Delta_d\,\sigma_e-x_e=a_e(\Delta_d-a_e)+b_e(\Delta_d-b_e)\ \ge\ 0,
\]
hence $x_e\le\Delta_d\,\sigma_e$, that is,
\begin{equation}\label{eq:db1-edge}
\frac{1}{x_e}\ \ge\ \frac{1}{\Delta_d\,\sigma_e}\qquad(e\in E(G)).
\end{equation}
Both summands above are nonnegative and $a_e,b_e>0$, so equality holds in
\eqref{eq:db1-edge} if and only if $a_e=b_e=\Delta_d$. Summing
\eqref{eq:db1-edge} over all edges gives
\begin{equation}\label{eq:db1-sum}
DRMS(G)=\sum_{e\in E(G)}\frac{1}{x_e}\ \ge\ \frac{1}{\Delta_d}\sum_{e\in E(G)}\frac{1}{\sigma_e}.
\end{equation}
Applying the Cauchy--Schwarz inequality to the sequences $\sqrt{\sigma_e}$ and
$1/\sqrt{\sigma_e}$ $(e\in E(G))$, we obtain
\begin{equation}\label{eq:db1-cs}
m^{2}=\Bigl(\sum_{e\in E(G)}1\Bigr)^{2}
\le\Bigl(\sum_{e\in E(G)}\sigma_e\Bigr)\Bigl(\sum_{e\in E(G)}\frac{1}{\sigma_e}\Bigr)
=2\,DB_1(G)\sum_{e\in E(G)}\frac{1}{\sigma_e}.
\end{equation}
Since $DB_1(G)>0$, this gives $\sum_{e}1/\sigma_e\ge m^{2}/\bigl(2\,DB_1(G)\bigr)$,
and combining with \eqref{eq:db1-sum} yields
\[
DRMS(G)\ \ge\ \frac{m^{2}}{2\,\Delta_d\,DB_1(G)} .
\]
For the equality case, write
\[
DRMS(G)-\frac{m^{2}}{2\Delta_d DB_1(G)}
=\sum_{e\in E(G)}\Bigl(\frac1{x_e}-\frac1{\Delta_d\sigma_e}\Bigr)
+\frac1{\Delta_d}\Bigl(\sum_{e\in E(G)}\frac1{\sigma_e}-\frac{m^{2}}{2DB_1(G)}\Bigr).
\]
The first sum is nonnegative by \eqref{eq:db1-edge} and the bracket is nonnegative
by \eqref{eq:db1-cs}; hence equality holds in the theorem if and only if both
vanish. If equality holds, then every term of the first sum vanishes, so equality
holds in \eqref{eq:db1-edge} for every edge, i.e.\ $a_e=b_e=\Delta_d$ for every
$e=uv\in E(G)$. As $G$ is connected, every vertex lies on an edge, so
$d_d\equiv\Delta_d$ and $G$ is domination-regular. Conversely, if $d_d\equiv k$
for some $k\ge1$, then $\Delta_d=k$, $\sigma_e=2k$ and $x_e=2k^{2}$ for every
edge, so $DB_1(G)=mk$ and
\[
\frac{m^{2}}{2\Delta_d DB_1(G)}=\frac{m^{2}}{2k\cdot mk}=\frac{m}{2k^{2}}=DRMS(G)
\]
\end{proof}

\begin{corollary}
\label{cor:db1-sharp}
If $G$ is $k$-domination-regular, then $DB_1(G)=mk$ and
\[
DRMS(G)=\frac{m}{2k^2}
=
\frac{m^2}{2\,\Delta_d\,DB_1(G)}.
\]
Hence the bound of Theorem~\ref{thm:db1} is attained.
\end{corollary}
Now, we define the constant $\kappa_d=\frac{\delta_d+\Delta_d}{\sqrt{\delta_d^2+\Delta_d^2}}\in(1,\sqrt2\,]$ which will be used subsequently.

\begin{lemma}\label{lem:kappa}
For every edge $e$,
\[
\kappa_d\ \le\ \frac{\sigma_e}{\sqrt{x_e}}\ \le\ \sqrt2 .
\]
The right equality holds if and only if $a_e=b_e$; the left one holds if and only if
$\{a_e,b_e\}=\{\delta_d,\Delta_d\}$ or $\delta_d=\Delta_d$.
\end{lemma}

\begin{proof}
Since $\sigma_e^{2}=x_e+2a_eb_e$ and $a_eb_e/x_e=\bigl(a_e/b_e+b_e/a_e\bigr)^{-1}$,
\[
\frac{\sigma_e^{2}}{x_e}=1+\frac{2}{a_e/b_e+b_e/a_e}.
\]
Put $r=\max\{a_e,b_e\}/\min\{a_e,b_e\}$, so that $a_e/b_e+b_e/a_e=f(r)$ with
$f(r)=r+1/r$ and $1\le r\le\Delta_d/\delta_d$. The function $f$ is strictly increasing on $[1,\infty)$, hence
\[
2=f(1)\ \le\ a_e/b_e+b_e/a_e\ \le\ f\!\left(\tfrac{\Delta_d}{\delta_d}\right)
=\frac{\Delta_d^{2}+\delta_d^{2}}{\delta_d\Delta_d}.
\]
Consequently
\[
\kappa_d^{2}=1+\frac{2\delta_d\Delta_d}{\delta_d^{2}+\Delta_d^{2}}\ \le\
\frac{\sigma_e^{2}}{x_e}\ \le\ 2 ,
\]
and taking square roots gives the claim. By strict monotonicity of $f$, the
upper bound is attained if and only if $r=1$, i.e.\ $a_e=b_e$, and the lower bound is
attained if and only if $r=\Delta_d/\delta_d$, i.e.\ $\{a_e,b_e\}=\{\delta_d,\Delta_d\}$
when $\delta_d<\Delta_d$; if $\delta_d=\Delta_d$ then $a_e=b_e$, both
bounds equal $\sqrt2$ and the lower equality holds trivially.
\end{proof}

\begin{theorem}
\label{thm:so}
For every connected graph $G$,
\[
DRMS(G)\ \ge\ \frac{m^{3}}{SO_d(G)^{2}}.
\]
Equality holds if and only if $x_e$ is constant over $E(G)$.
\end{theorem}

\begin{proof}
By Lemma \ref{lem:delta1} we have $\delta_d\ge1$, so $x_e>0$ for every edge $e$. Apply
H\"older's inequality with the conjugate exponents $3$ and $3/2$ to the
positive sequences $x_e^{-1/3}$ and $x_e^{1/3}$ $(e\in E(G))$, we obtain
\[
m=\sum_{e\in E(G)}x_e^{-1/3}x_e^{1/3}
\le\Bigl(\sum_{e\in E(G)}x_e^{-1}\Bigr)^{1/3}\Bigl(\sum_{e\in E(G)}x_e^{1/2}\Bigr)^{2/3}
=DRMS(G)^{1/3}\,SO_d(G)^{2/3}.
\]
Cubing gives $m^{3}\le DRMS(G)\,SO_d(G)^{2}$, and since $SO_d(G)>0$ the
claimed inequality follows. Equality in H\"older's inequality holds if and only if the
sequences $\bigl(x_e^{-1/3}\bigr)^{3}=x_e^{-1}$ and
$\bigl(x_e^{1/3}\bigr)^{3/2}=x_e^{1/2}$ are proportional, that is,
$x_e^{-1}=\lambda\,x_e^{1/2}$ for some constant $\lambda>0$ and all $e$. This
means $x_e^{3/2}=1/\lambda$, so $x_e$ is constant over $E(G)$. Conversely, if
$x_e\equiv x$, then $DRMS(G)=m/x$ and $SO_d(G)=m\sqrt{x}$, hence
$m^{3}/SO_d(G)^{2}=m/x=DRMS(G)$.
\end{proof}

For every edge $e$ put $y_e:=x_e^{-1/2}$, so that
$DRMS(G)=\sum_{e}y_e^{2}$ and ${}^{m}SO_d(G)=\sum_{e}y_e$. Since
$2\delta_d^{2}\le x_e\le 2\Delta_d^{2}$, we have
\[
y_-\ :=\ \frac{1}{\sqrt2\,\Delta_d}\ \le\ y_e\ \le\ \frac{1}{\sqrt2\,\delta_d}\ =:\ y_+ .
\]

\begin{theorem}\label{thm:mdso}
For every connected graph $G$,
\[
\frac{{}^{m}SO_d(G)^{2}}{m}\ \le\ DRMS(G)\ \le\
\frac{\delta_d+\Delta_d}{\sqrt2\,\delta_d\Delta_d}\,{}^{m}SO_d(G)-\frac{m}{2\delta_d\Delta_d}.
\]
Equality holds on the left if and only if $x_e$ is constant over $E(G)$, and
on the right if and only if $G$ is domination-regular.
\end{theorem}

\begin{proof}
By Lemma \ref{lem:delta1}, $x_e>0$ for every edge $e$, so $y_e:=x_e^{-1/2}$ is well defined,
and $DRMS(G)=\sum_{e}y_e^{2}$, ${}^{m}SO_d(G)=\sum_{e}y_e$. Put
\[
y_-=\frac{1}{\sqrt2\,\Delta_d},\qquad y_+=\frac{1}{\sqrt2\,\delta_d}.
\]
By \eqref{eq:basic-ineq}, $2\delta_d^{2}\le x_e\le2\Delta_d^{2}$, hence $y_-\le y_e\le y_+$ for
every edge $e$.\\
By the Cauchy--Schwarz inequality,
$\bigl(\sum_e y_e\bigr)^{2}\le m\sum_e y_e^{2}$, which is the left inequality;
equality holds if and only if $y_e$ is constant, that is, if and only if $x_e$ is constant. For the
right inequality, $(y_e-y_-)(y_+-y_e)\ge0$ gives
$y_e^{2}\le(y_-+y_+)\,y_e-y_-y_+$ for every edge. Summing over $E(G)$,
\[
DRMS(G)\le(y_-+y_+)\,{}^{m}SO_d(G)-m\,y_-y_+,
\]
and inserting $y_-+y_+=\dfrac{\delta_d+\Delta_d}{\sqrt2\,\delta_d\Delta_d}$ and
$y_-y_+=\dfrac{1}{2\delta_d\Delta_d}$ gives the claim. The difference between
the two sides equals $\sum_{e}(y_e-y_-)(y_+-y_e)$, a sum of nonnegative terms,
so equality holds if and only if $y_e\in\{y_-,y_+\}$ for every edge, that is,
$x_e\in\{2\Delta_d^{2},2\delta_d^{2}\}$. Since $\delta_d\le a_e,b_e\le\Delta_d$,
$x_e=2\delta_d^{2}$ forces $a_e=b_e=\delta_d$ and $x_e=2\Delta_d^{2}$ forces
$a_e=b_e=\Delta_d$. Thus equality implies that both endpoints of every edge
have the same domination degree, so $d_d$ is constant along every path; as $G$
is connected, $d_d$ is constant on $V(G)$ and $G$ is domination-regular.
Conversely, if $d_d\equiv k$, then $\delta_d=\Delta_d=k$ and
$y_e=y_-=y_+=\dfrac1{\sqrt2\,k}$ for every edge, so each term
$(y_e-y_-)(y_+-y_e)$ vanishes and equality holds; both sides equal
$m/(2k^{2})$.
\end{proof}

\begin{theorem}\label{thm:deso}
Let \(G\) be a connected graph. Then
\[
\frac{\kappa_d \, m^2}{DESO(G)} \le DRMS(G) \le \frac{\sqrt{2} \, c_d \, m^2}{DESO(G)},
\]
where, $
\kappa_d = \frac{\delta_d + \Delta_d}{\sqrt{\delta_d^2 + \Delta_d^2}}, c_d = \left(\frac{\delta_d^2 + \Delta_d^2}{2\delta_d \Delta_d}\right)^2.
$ \\
Equality on the left holds if and only if every edge joins a vertex of
domination degree \(\delta_d\) to one of domination degree \(\Delta_d\)
(or \(\delta_d = \Delta_d\)); equality on the right holds if and only if \(G\)
is domination-regular.
\end{theorem}
\begin{proof}
By Lemma~\ref{lem:kappa}, \(\kappa_d \sqrt{x_e} \le \sigma_e \le \sqrt{2} \sqrt{x_e}\)
for every edge \(e \in E(G)\). Multiplying by \(\sqrt{x_e}\) and summing over
all edges gives
\begin{equation}\label{eq:deso-i}
\kappa_d \, DF^*(G) \le DESO(G) \le \sqrt{2} \, DF^*(G).
\end{equation}
By Theorems~\ref{thm:cauchy} and~\ref{thm:kantorovich},
\begin{equation}\label{eq:deso-ii}
m^2 \le DF^*(G) \cdot DRMS(G) \le c_d \, m^2.
\end{equation}\\
\emph{Lower bound.} Multiplying the left inequalities of
\eqref{eq:deso-i} and \eqref{eq:deso-ii},
\[
DESO(G) \cdot DRMS(G) \ge \kappa_d \, DF^*(G) \cdot DRMS(G) \ge \kappa_d \, m^2.
\]
Dividing by \(DESO(G) > 0\) gives \(DRMS(G) \ge \kappa_d m^2 / DESO(G)\).\\
\emph{Upper bound.} Multiplying the right inequalities,
\[
DESO(G) \cdot DRMS(G) \le \sqrt{2} \, DF^*(G) \cdot DRMS(G) \le \sqrt{2} \, c_d \, m^2.
\]
Dividing by \(DESO(G) > 0\) gives
\(DRMS(G) \le \sqrt{2} \, c_d m^2 / DESO(G)\).\\
\emph{Equality, left.} It requires equality in both \eqref{eq:deso-i} (left)
and \eqref{eq:deso-ii} (left). By the equality condition of
Lemma~\ref{lem:kappa}, the former forces \(\sigma_e/\sqrt{x_e} = \kappa_d\) on
every edge, that is, \(\{a_e, b_e\} = \{\delta_d, \Delta_d\}\) or
\(\delta_d = \Delta_d\). Conversely, if every edge satisfies
\(\{a_e, b_e\} = \{\delta_d, \Delta_d\}\), then
\[
x_e = a_e^2 + b_e^2 = \delta_d^2 + \Delta_d^2
\]
is constant over \(E(G)\), so equality holds in the Cauchy--Schwarz bound
\eqref{eq:deso-ii} (left). Hence equality on the left holds under the stated
condition.\\
\emph{Equality, right.} It requires equality in both \eqref{eq:deso-i} (right)
and \eqref{eq:deso-ii} (right). By the equality condition of
Lemma~\ref{lem:kappa}, the former forces \(\sigma_e = \sqrt{2}\sqrt{x_e}\), that
is, \(a_e = b_e\) on every edge. Since \(G\) is connected, this implies
\(d_d \equiv k\) on \(V(G)\), that is, \(G\) is domination-regular.
Conversely, if \(d_d \equiv k\), then \(\delta_d = \Delta_d = k\), so
\[
c_d = \left(\frac{k^2 + k^2}{2k^2}\right)^2 = 1,
\]
and \(a_e = b_e\) for every edge, so all inequalities become equalities.
\end{proof}

\section{Closed Forms for Standard Graph Families}
In this section, we compute the exact value of \(DRMS(G)\) for several
standard graph families. For each family, we first determine the minimal
dominating sets of the graph, then compute the domination degrees of its
vertices, and finally evaluate \(DRMS\) by summing over the edge set.
\begin{proposition}\label{prop:kn}
For every \(n \ge 2\), 
\[
DRMS(K_n) = \dfrac{n(n-1)}{4}.
\]
where \(K_n\) denotes the complete graph on \(n\) vertices.
\end{proposition}

\begin{proof}
Every singleton \(\{v\}\) is a minimal dominating set of \(K_n\): it is
dominating because \(v\) is adjacent to all other vertices, and it is minimal
because it contains no proper subset. Conversely, any dominating set \(S\) with
\(|S| \ge 2\) is not minimal, since \(S \setminus \{u\}\) still dominates
\(K_n\) for any \(u \in S\). Hence the minimal dominating sets of \(K_n\) are
exactly the singletons, so \(d_d(v) = 1\) for every \(v \in V(K_n)\).\\
Since \(d_d^2(u) + d_d^2(v) = 2\) for every edge \(uv\), and
\(m = \binom{n}{2} = n(n-1)/2\), we obtain
\[
DRMS(K_n) = \sum_{uv \in E(K_n)} \frac{1}{2} = \frac{m}{2} = \frac{n(n-1)}{4}. \qedhere
\]
\end{proof}

\begin{proposition}\label{prop:star}
For every \(n \ge 2\), \[DRMS(S_n) = \dfrac{n-1}{2}.\]
where \(S_n = K_{1, n-1}\) denotes the star on \(n\) vertices.
\end{proposition}

\begin{proof}
Let \(c\) be the center and \(l_1, \ldots, l_{n-1}\) the leaves of \(S_n\).
We claim that the minimal dominating sets of \(S_n\) are exactly
\[
S_c := \{c\}, \qquad S_L := \{l_1, \ldots, l_{n-1}\}.
\]
\emph{Both are minimal dominating sets.} The singleton \(\{c\}\) dominates
\(S_n\) (every leaf is adjacent to \(c\)) and is minimal (it has no proper
subset). The set \(S_L\) dominates \(S_n\) (each leaf dominates itself, and
\(c\) is adjacent to every leaf in \(S_L\)); it is minimal because removing
any leaf \(l_i\) leaves \(l_i\) undominated (the leaves are pairwise
non-adjacent and \(c \notin S_L\)).\\
\emph{No other minimal dominating set exists.} Let \(S\) be a minimal
dominating set of \(S_n\). If \(c \in S\), then \(S = \{c\}\) by minimality
(since \(\{c\}\) already dominates \(S_n\)). If \(c \notin S\), then every
leaf \(l_i\) must belong to \(S\) (a leaf is adjacent only to \(c\), which is
not in \(S\)); hence \(S \supseteq S_L\), and minimality forces \(S = S_L\).\\
Consequently \(d_d(c) = 1\) (only \(S_c\) contains \(c\)) and
\(d_d(l_i) = 1\) (only \(S_L\) contains \(l_i\)), so \(d_d(v) = 1\) for every
\(v \in V(S_n)\). Since \(S_n\) has \(m = n-1\) edges and every edge
contributes \(1/(1^2 + 1^2) = 1/2\),
\[
DRMS(S_n) = \frac{m}{2} = \frac{n-1}{2}. \qedhere
\]
\end{proof}

\begin{proposition}\label{prop:double-star}
For every \(p, q \ge 1\), 
\[
DRMS(S_{p,q}) = \frac{p+q+1}{8}.
\]
where \(S_{p,q}\) denotes the double-star graph formed by joining the centres
of two stars \(K_{1,p}\) and \(K_{1,q}\).
\end{proposition}

\begin{proof}
Let \(c_1, c_2\) be the two centers, \(l_1, \ldots, l_p\) the leaves adjacent
to \(c_1\), and \(l'_1, \ldots, l'_q\) the leaves adjacent to \(c_2\). We claim
that the minimal dominating sets of \(S_{p,q}\) are exactly
\[
S_1 = \{c_1, c_2\}, \quad
S_2 = \{c_1\} \cup \{l'_1, \ldots, l'_q\}, \quad
S_3 = \{c_2\} \cup \{l_1, \ldots, l_p\}, \quad
S_4 = \{l_1, \ldots, l_p, l'_1, \ldots, l'_q\}.
\]
Each of these dominates \(S_{p,q}\) and is minimal (removing any vertex
leaves either a leaf undominated or a center that is the only dominator of
its leaves).\\
Conversely, let \(S\) be a minimal dominating set. If \(c_1, c_2 \in S\),
then \(S = S_1\) by minimality. If \(c_1 \in S\) but \(c_2 \notin S\), every
leaf \(l'_j\) must belong to \(S\) (it is adjacent only to \(c_2\), which is
not in \(S\)), so \(S \supseteq \{c_1\} \cup \{l'_1, \ldots, l'_q\} = S_2\);
minimality forces \(S = S_2\). The case \(c_2 \in S\), \(c_1 \notin S\) is
symmetric and gives \(S = S_3\). Finally, if \(c_1, c_2 \notin S\), then every
leaf must belong to \(S\), so \(S = S_4\). Hence there are exactly four
minimal dominating sets.\\
Counting incidences, each vertex lies in exactly two of them:
\(c_1 \in S_1, S_2\); \(c_2 \in S_1, S_3\); each \(l_i \in S_3, S_4\); each
\(l'_j \in S_2, S_4\). Therefore \(d_d(v) = 2\) for every \(v \in V(S_{p,q})\).
Since \(S_{p,q}\) has \(m = p + q + 1\) edges and every edge contributes
\(1/(2^2 + 2^2) = 1/8\),
\[
DRMS(S_{p,q}) = \frac{m}{8} = \frac{p+q+1}{8}. \qedhere
\]
\end{proof}

\begin{proposition}\label{prop:kmn}
For every \(m, n \ge 2\),
\[
DRMS(K_{m,n}) = \frac{mn}{(m+1)^2 + (n+1)^2}.
\]
where \(K_{m,n}\) denotes the complete bipartite graph with parts of sizes
\(m\) and \(n\).
\end{proposition}

\begin{proof}
Let \(V_1, V_2\) be the two parts of \(K_{m,n}\), with \(|V_1| = m\) and
\(|V_2| = n\). We first determine the minimal dominating sets of \(K_{m,n}\).
We claim they are exactly:
\[
V_1, \qquad V_2, \qquad \{v, w\} \text{ with } v \in V_1,\ w \in V_2.
\]
\emph{Each is a minimal dominating set.} The set \(V_1\) dominates \(K_{m,n}\)
(every vertex of \(V_2\) is adjacent to all of \(V_1\)) and is minimal
(removing any \(v \in V_1\) leaves \(v\) undominated, since \(v\) is adjacent
only to \(V_2\), which is disjoint from \(V_1\)). The same argument applies
to \(V_2\). For \(\{v, w\}\), domination holds because every vertex of \(V_1\)
is adjacent to \(w\) and every vertex of \(V_2\) is adjacent to \(v\).
Minimality holds because neither \(\{v\}\) nor \(\{w\}\) dominates: \(V_1
\setminus \{v\}\) is undominated by \(\{w\}\) (its vertices are adjacent only
to \(V_2\), and \(w \in V_2\) alone does not cover them), and symmetrically
for \(\{v\}\).\\
\emph{No other minimal dominating set exists.} Let \(S\) be a minimal
dominating set. If \(S \subseteq V_1\), then for every \(v \in V_1 \setminus S\)
to be dominated, some vertex of \(S\) must be adjacent to it; but no vertex
of \(V_1\) is adjacent to \(v\). Hence \(S = V_1\). Symmetrically for
\(S \subseteq V_2\). If \(S\) meets both \(V_1\) and \(V_2\), say \(v \in S
\cap V_1\) and \(w \in S \cap V_2\), then \(\{v, w\} \subseteq S\) already
dominates \(K_{m,n}\), so minimality forces \(S = \{v, w\}\).\\
\emph{Domination degrees.} A vertex \(v \in V_1\) lies in \(V_1\) and in the
\(n\) pairs \(\{v, w\}\) with \(w \in V_2\), hence \(d_d(v) = n + 1\).
Symmetrically, \(d_d(w) = m + 1\) for \(w \in V_2\).\\
\emph{DRMS.} Every edge \(vw\) with \(v \in V_1\), \(w \in V_2\) contributes
\[
\frac{1}{d_d^2(v) + d_d^2(w)} = \frac{1}{(n+1)^2 + (m+1)^2},
\]
and there are \(mn\) such edges. Hence
\[
DRMS(K_{m,n}) = \frac{mn}{(m+1)^2 + (n+1)^2}. \qedhere
\]
\end{proof}

\begin{definition}\label{def:kpartite}
For \(k \ge 2\) and \(r_1, \ldots, r_k \ge 2\), the \emph{complete
\(k\)-partite graph} \(K_{r_1, \dots, r_k}\) has vertex set partitioned into
\(k\) parts \(V_1, \ldots, V_k\) of sizes \(r_1, \ldots, r_k\), with two
vertices adjacent if and only if they lie in different parts. We write
\(N = \sum_{i=1}^k r_i\).
\end{definition}

\begin{lemma}\label{lem:kpartite-mds}
Let \(G = K_{r_1, \dots, r_k}\) with \(k \ge 2\) and all \(r_i \ge 2\). The
minimal dominating sets of \(G\) are exactly the parts \(V_l\) and the pairs
\(\{u, w\}\) with \(u, w\) in different parts.
\end{lemma}

\begin{proof}
Let \(V_1, \ldots, V_k\) be the parts.\\
\emph{Domination.} If \(u \in V_i\), \(w \in V_j\), \(i \neq j\), then
\(\{u, w\}\) dominates \(G\): vertices of \(V_i\) are adjacent to \(w\),
vertices of \(V_j\) to \(u\), and all other vertices to both. Each \(V_l\)
dominates \(G\), since every vertex outside \(V_l\) is adjacent to all of
\(V_l\).\\
\emph{Minimality.} Removing \(u \in V_l\) from \(V_l\) leaves \(u\)
undominated. For \(\{u, w\}\) as above, \(\{w\}\) does not dominate the
vertices of \(V_j \setminus \{w\}\) (nonempty because \(r_j \geq 2\)), and
\(\{u\}\) does not dominate \(V_i \setminus \{u\}\) (because \(r_i \geq 2\)).\\
\emph{Completeness.} Let \(S\) be a minimal dominating set. If \(S\) meets two
different parts, it contains a dominating pair, so \(S\) equals that pair by
minimality. Otherwise \(S \subseteq V_l\) for some \(l\); a vertex of \(V_l\)
is dominated by \(S\) only if it lies in \(S\) (it is not adjacent to anything
in \(V_l\)), hence \(S \supseteq V_l\) and \(S = V_l\).
\end{proof}

\begin{lemma}\label{lem:kpartite-dd}
Under the notation of Lemma~\ref{lem:kpartite-mds}, for every vertex \(v\) in
the \(l\)-th part,
\[
d_d(v) = N - r_l + 1,
\]
where \(N = \sum_i r_i\).
\end{lemma}

\begin{proof}
By Lemma~\ref{lem:kpartite-mds}, the minimal dominating sets of \(G\) are the
parts \(V_1, \ldots, V_k\) and the pairs \(\{u, w\}\) with \(u, w\) in
different parts.\\
A vertex \(v \in V_l\) lies in the set \(V_l\) and in the \(N - r_l\) pairs
\(\{v, w\}\) with \(w \notin V_l\). No other minimal dominating set contains
\(v\). Hence \(d_d(v) = 1 + (N - r_l) = N - r_l + 1\).
\end{proof}

\begin{theorem}\label{thm:kpartite}
Let \(G = K_{r_1, \dots, r_k}\) with \(k \ge 2\) and all \(r_i \ge 2\), and
let \(N = \sum_i r_i\). Then
\[
DRMS(G) = \sum_{1 \le l < j \le k} \frac{r_l r_j}{(N - r_l + 1)^2 + (N - r_j + 1)^2}.
\]
\end{theorem}
\begin{proof}
Let $V_1,\dots,V_k$ be the parts. We claim that the minimal dominating sets
are exactly the parts $V_l$ and the pairs $\{u,w\}$ with $u,w$ in different
parts.\\
\emph{Domination.} If $u\in V_i$, $w\in V_j$, $i\ne j$, then $\{u,w\}$
dominates $G$: vertices of $V_i$ are adjacent to $w$, vertices of $V_j$ to
$u$, and all other vertices to both. Each $V_l$ dominates $G$, since every
vertex outside $V_l$ is adjacent to all of $V_l$.\\
\emph{Minimality.} Removing $u\in V_l$ from $V_l$ leaves $u$ undominated.
For $\{u,w\}$ as above, $\{w\}$ does not dominate the vertices of $V_j\setminus\{w\}$
(nonempty because $r_j\ge2$), and $\{u\}$ does not dominate
$V_i\setminus\{u\}$ (because $r_i\ge2$).\\
\emph{Completeness.} Let $S$ be a minimal dominating set. If $S$ meets two
different parts, it contains a dominating pair, so $S$ equals that pair by
minimality. Otherwise $S\subseteq V_l$ for some $l$; a vertex of $V_l$ is
dominated by $S$ only if it lies in $S$ (it is not adjacent to anything in
$V_l$), hence $S\supseteq V_l$ and $S=V_l$.\\
\emph{Counting.} A vertex $v\in V_l$ lies in the set $V_l$ and in the $N-r_l$
pairs $\{v,w\}$ with $w\notin V_l$, so $d_d(v)=N-r_l+1$. There are $r_lr_j$
edges between $V_l$ and $V_j$, each contributing
$1/\bigl((N-r_l+1)^2+(N-r_j+1)^2\bigr)$.
\end{proof}

\begin{definition}\label{def:wheel}
For \(n \ge 4\), the \emph{wheel graph} \(W_n\) consists of a cycle
\(C_{n-1}\) with vertices \(v_1, \dots, v_{n-1}\) together with a central
vertex \(v_0\) adjacent to every vertex of the cycle.
\end{definition}

\begin{definition}
\label{def:tk}
For $k\ge 3$, fix a vertex $x$ of the cycle $C_k$ and define
\[
t_k = \mathrm{card}\{S \subseteq V(C_k) : S \text{ is a minimal dominating set of } C_k,\ x \in S\}.
\]
Since $C_k$ is vertex-transitive, $t_k$ does not depend on the choice of $x$.
\end{definition}

\begin{lemma}
\label{lem:wheel}
Let $n\ge 4$. For every $S\subseteq V(C_{n-1})$, the following are equivalent:
\begin{enumerate}
\item $S$ is a minimal dominating set of $W_n$ with $v_0\notin S$;
\item $S$ is a minimal dominating set of $C_{n-1}$.
\end{enumerate}
\end{lemma}

\begin{proof}
Let $S\subseteq V(C_{n-1})$, so $v_0\notin S$. For every cycle vertex $v_i$ we have
\begin{equation}\label{eq:nbr}
  N_{W_n}[v_i]\cap S \;=\; N_{C_{n-1}}[v_i]\cap S,
\end{equation}
because the only additional neighbour of $v_i$ in $W_n$ is $v_0\notin S$.\\
Domination.
By \eqref{eq:nbr}, a cycle vertex is dominated by $S$ in $W_n$ if and only if it is dominated by $S$ in $C_{n-1}$. The centre $v_0$ is adjacent to every vertex of $C_{n-1}$, so it is dominated by every nonempty $S\subseteq V(C_{n-1})$. Every dominating set of $C_{n-1}$ is nonempty. Hence
\[
  S \text{ dominates } W_n \text{ if and only if } S \text{ dominates } C_{n-1}.
\]
\medskip
Minimality.
Assume that $S$ dominates both graphs. By Lemma~\ref{lem:private} (a dominating set is minimal if and only if every vertex of it has a private neighbour), it suffices to compare the private neighbours of each $u\in S$ in the two graphs. Let $w\in V(W_n)$ satisfy $N_{W_n}[w]\cap S=\{u\}$.\\
\begin{itemize}
  \item If $w\in V(C_{n-1})$, then by \eqref{eq:nbr} we get $N_{C_{n-1}}[w]\cap S=\{u\}$. Thus $w$ is also a private neighbour of $u$ in $C_{n-1}$. Conversely, every private neighbour of $u$ in $C_{n-1}$ is a private neighbour of $u$ in $W_n$, again by \eqref{eq:nbr}.
  \item If $w=v_0$, then $N_{W_n}[v_0]\cap S=S$, so $S=\{u\}$. A one-element dominating set is always minimal, since $\emptyset$ is not dominating. Hence $S$ is minimal in $W_n$. Moreover, $S=\{u\}$ dominates $C_{n-1}$ and is therefore minimal there as well. (This case occurs only for $n=4$, i.e.\ $W_4=K_4$.)
\end{itemize}
Consequently, if $|S|\ge 2$, then $v_0$ is never a private neighbour, and $u\in S$ has a private neighbour in $W_n$ if and only if it has one in $C_{n-1}$. If $|S|=1$, then $S$ is minimal in both graphs. In all cases,
\[
  S \text{ is minimal in } W_n \text{ if and only if } S \text{ is minimal in } C_{n-1}.
\]
Together with the domination equivalence, this proves the lemma.
\end{proof}

\begin{theorem}
\label{thm:wheel-dd}
Let \(W_n\) be the wheel graph with centre \(v_0\) and cycle vertices
\(v_1, \ldots, v_{n-1}\). For \(n \ge 4\),
\[
d_d(v_0) = 1, \qquad d_d(v_i) = t_{n-1} \quad (1 \le i \le n-1).
\]
\end{theorem}

\begin{proof}
The set $\{v_0\}$ is the only minimal dominating set of $W_n$ containing
$v_0$, so $d_d(v_0)=1$. Let $v_i$ be a cycle vertex with $1\le i\le n-1$,
and let $S$ be a minimal dominating set of $W_n$ containing $v_i$. If
$v_0\in S$, then by minimality of $S$ we must have $S=\{v_0\}$, which
contradicts $v_i\in S$. Hence $v_0\notin S$, and by Lemma~\ref{lem:wheel}
$S$ is a minimal dominating set of $C_{n-1}$ containing $v_i$. Conversely,
every minimal dominating set of $C_{n-1}$ containing $v_i$ is a minimal
dominating set of $W_n$ containing $v_i$. Therefore the number of such sets
is exactly $t_{n-1}$.
\end{proof}

\begin{theorem}
\label{thm:wheel}
For $n\ge 4$,
\[
DRMS(W_n)
=(n-1)\left[
\frac{1}{1+t_{n-1}^2}
+\frac{1}{2t_{n-1}^2}
\right].
\]
\end{theorem}

\begin{proof}
The wheel $W_n$ has $2(n-1)$ edges: $n-1$ spokes joining $v_0$ to the cycle
vertices, and $n-1$ edges of the cycle $C_{n-1}$. Using the domination
degrees from Theorem~\ref{thm:wheel-dd},
\[
DRMS(W_n)
=(n-1)\cdot\frac{1}{1+t_{n-1}^2}
+(n-1)\cdot\frac{1}{2t_{n-1}^2}.
\]
\end{proof}

\begin{example}
For $W_6$, we have $t_5=2$, hence $d_d(v_0)=1$ and $d_d(v_i)=2$ for
$1\le i\le 5$, and
\[
DRMS(W_6)=5\left(\frac15+\frac18\right)=\frac{13}{8}=1.625.
\]
\end{example}

\begin{definition}\label{def:book}
For \(r \ge 2\), the \emph{book graph} \(B_r\) consists of two adjacent
\emph{hub} vertices \(a, b\), together with \(2r\) \emph{page} vertices
\(p_1, \ldots, p_r\) and \(p'_1, \ldots, p'_r\), where \(p_i\) is adjacent
to \(a\), \(p'_i\) is adjacent to \(b\), and \(p_i\) is adjacent to \(p'_i\)
for each \(i\). The graph has \(2(r+1)\) vertices and \(3r+1\) edges.
\end{definition}

\begin{lemma}\label{lem:book-mds}
\(B_r\) has exactly \(2^r + 3\) minimal dominating sets, namely
\[
\{a,b\}; \quad \{a,p_1,\dots,p_r\}; \quad \{a,p_1',\dots,p_r'\};
\]
\[
\{b,p_1,\dots,p_r\}; \quad \{b,p_1',\dots,p_r'\};
\]
and the \(2^r - 2\) sets
\[
S_J = \{p_i : i \in J\} \cup \{p_i' : i \notin J\},
\qquad \varnothing \neq J \subsetneq [r].
\]
\end{lemma}

\begin{proof}
Note \(N[a] = \{a, b, p_1, \dots, p_r\}\), \(N[b] = \{b, a, p_1', \dots, p_r'\}\),
\(N[p_i] = \{p_i, a, p_i'\}\), \(N[p_i'] = \{p_i', b, p_i\}\). Let \(S\) be
minimal dominating.

\emph{Case \(a, b \in S\).} \(\{a,b\}\) already dominates \(G\), so
\(S = \{a,b\}\); it is minimal since \(\{a\}\) misses \(p_i'\) and \(\{b\}\)
misses \(p_i\).

\emph{Case \(a \in S\), \(b \notin S\).} Every \(p_i'\) must be dominated, so
\(S \cap \{p_i, p_i'\} \neq \varnothing\) for all \(i\). A private neighbour of
\(a\) lies in \(N[a]\): the vertex \(a\) itself works if and only if \(S\) contains no
\(p_i\); the vertex \(b\) works if and only if \(S\) contains no \(p_i'\); a vertex
\(p_i\) never works since \(S \cap \{p_i, p_i'\} \neq \varnothing\) and
\(\{p_i, p_i'\} \subseteq N[p_i]\) with \(a \in S\). Hence either no
\(p_i \in S\), forcing \(S \supseteq \{a, p_1', \dots, p_r'\}\), or no
\(p_i' \in S\), forcing \(S \supseteq \{a, p_1, \dots, p_r\}\). These sets are
dominating, so \(S\) equals one of them, and both are minimal: in
\(\{a, p_1', \dots, p_r'\}\) the vertex \(a\) is its own private neighbour and
each \(p_i'\) is its own private neighbour; in \(\{a, p_1, \dots, p_r\}\) the
vertex \(a\) has the private neighbour \(b\) and \(p_i\) has the private
neighbour \(p_i'\).

\emph{Case \(b \in S\), \(a \notin S\).} Symmetric: two sets.

\emph{Case \(a, b \notin S\).} To dominate \(p_i\) and \(p_i'\) we need
\(S \cap \{p_i, p_i'\} \neq \varnothing\) for every \(i\). Suppose
\(p_i, p_i' \in S\) for some \(i\). A private neighbour of \(p_i\) in
\(N[p_i] = \{p_i, a, p_i'\}\) can only be \(a\) (the others are dominated by
\(p_i'\)), which requires \(p_i\) to be the only \(p\) in \(S\); likewise
\(p_i'\) needs \(b\), i.e.\ \(p_i'\) is the only \(p'\) in \(S\). Then for
\(j \neq i\) (such \(j\) exists as \(r \geq 2\)) neither \(p_j\) nor
\(p_j'\) lies in \(S\), a contradiction. So exactly one of \(p_i, p_i'\) lies
in \(S\) for each \(i\): \(S = S_J\). Domination of \(a\) and \(b\) requires
\(J \neq \varnothing\) and \(J \neq [r]\). Conversely each such \(S_J\) is
dominating and minimal (\(p_i\) has the private neighbour \(p_i'\) and vice
versa).
\end{proof}

\begin{lemma}\label{lem:book-dd}
For every \(r \geq 2\),
\[
d_d(a) = d_d(b) = 3, \qquad d_d(p_i) = d_d(p_i') = 2^{r-1} + 1 =: q.
\]
\end{lemma}

\begin{proof}
By Lemma~\ref{lem:book-mds}, \(a\) lies in \(\{a,b\}\) and in the two sets of
the second case: \(d_d(a) = 3\); same for \(b\). A vertex \(p_i\) lies in one
set of the second case, one of the third, and in \(S_J\) for \(i \in J \neq [r]\),
i.e.\ \(2^{r-1} - 1\) sets: \(d_d(p_i) = 2^{r-1} + 1\). By symmetry,
\(d_d(p_i') = 2^{r-1} + 1\).
\end{proof}

\begin{theorem}\label{thm:book}
For every \(r \geq 2\),
\[
DRMS(B_r) = \frac{1}{18} + \frac{2r}{9 + q^2} + \frac{r}{2q^2},
\qquad q = 2^{r-1} + 1.
\]
\end{theorem}

\begin{proof}
By Lemma~\ref{lem:book-dd}, the edge classes are \(ab\) (weight \(1/18\)),
\(2r\) hub--page edges (weight \(1/(9 + q^2)\)) and \(r\) edges
\(p_i p_i'\) (weight \(1/(2q^2)\)). Summing over all edges,
\[
DRMS(B_r) = \frac{1}{18} + \frac{2r}{9 + q^2} + \frac{r}{2q^2}.
\]
\end{proof}

\begin{definition}\label{def:butterfly}
For \(r, s \ge 1\), the \emph{butterfly graph} \(B_{r,s}\) is formed by
taking a central vertex \(c\), attaching to it \(r\) triangles
\(c x_i y_i c\) that are pairwise disjoint except for \(c\), and attaching
to it \(s\) pendant vertices \(z_1, \ldots, z_s\). The graph has
\(2r + s + 1\) vertices and \(3r + s\) edges.
\end{definition}

\begin{theorem}\label{thm:butterfly}
The minimal dominating sets of $B_{r,s}$ are $\{c\}$ and the $2^r$ sets
consisting of $\{z_1,\dots,z_s\}$ together with exactly one of $x_i,y_i$ for
each $i$. Therefore $d_d(c)=1$, $d_d(x_i)=d_d(y_i)=2^{r-1}$, $d_d(z_j)=2^r$ and
\[
DRMS(B_{r,s})=\frac{s}{1+4^{r}}+\frac{2r}{1+4^{r-1}}+\frac{r}{2\cdot4^{r-1}}.
\]
\end{theorem}
\begin{proof}
If $c\in S$ then $S=\{c\}$ because $\{c\}$ dominates $G$ (and it is minimal,
since $z_1$ is not dominated without $c$ or $z_1$). Let $c\notin S$. Since
$N[z_j]=\{z_j,c\}$, all $z_j\in S$ (this also dominates $c$, as $s\ge1$). For
each triangle, $S\cap\{x_i,y_i\}\ne\varnothing$. If both $x_i,y_i\in S$, then
$x_i$ has no private neighbour (its closed neighbourhood $\{x_i,y_i,c\}$ is
dominated also by $y_i$ or $z_1$), contradicting Lemma~\ref{lem:private}. So
exactly one is chosen. Conversely, such a set is dominating, $z_j$ is its own
private neighbour and $x_i$ (resp.\ $y_i$) has the private neighbour $y_i$
(resp.\ $x_i$). Counting: $c$ lies in one set; each $z_j$ in all $2^r$ sets of
the second type; each $x_i$ in half of them. The edge classes are $s$ edges
$cz_j$ with $x_e=1+4^r$, $2r$ edges $cx_i,cy_i$ with $x_e=1+4^{r-1}$, and $r$
edges $x_iy_i$ with $x_e=2\cdot4^{r-1}$.
\end{proof}

\section{$DRMS$ for the Corona Product}
In this final section, we study the corona product \(G \circ F\) and derive
closed-form expressions for \(DRMS\) in two important cases: the corona with a
complete graph \(K_{n_2}\), and the corona with an empty graph
\(\overline{K_{n_2}}\). In both cases, the graph turns out to be
domination-regular, which allows us to express \(DRMS\) directly in terms of
the number of edges.
\begin{definition}\label{def:corona}
Let \(G\) be a graph of order \(n_1\) and size \(m_1\), and let \(F\) be any
graph. The \emph{corona} \(G \circ F\) is obtained by taking, for each vertex
\(v \in V(G)\), one copy \(F^v\) of \(F\), with the copies being pairwise
disjoint, and joining \(v\) to every vertex of \(F^v\). The resulting graph has
\[
|V(G \circ F)| = n_1 (1 + |V(F)|), \qquad
|E(G \circ F)| = m_1 + n_1 \cdot |E(F)| + n_1 \cdot |V(F)|.
\]
\end{definition}

\begin{theorem}\label{thm:corona}
Let $G$ be connected, $n_2\ge1$.
\begin{enumerate}
\item $G\circ K_{n_2}$ has exactly $(n_2+1)^{n_1}$ minimal dominating sets and
is $k$-domination-regular with $k=(n_2+1)^{n_1-1}$. It has
$m_1+n_1n_2(n_2+1)/2$ edges, so
\[
DRMS(G\circ K_{n_2})=
\frac{m_1+n_1n_2(n_2+1)/2}{2\,(n_2+1)^{2(n_1-1)}}.
\]
\item $G\circ\overline{K_{n_2}}$ has exactly $2^{n_1}$ minimal dominating sets
and is $2^{n_1-1}$-domination-regular, with $m_1+n_1n_2$ edges:
\[
DRMS(G\circ\overline{K_{n_2}})=\frac{m_1+n_1n_2}{2^{\,2n_1-1}}.
\]
\end{enumerate}
\end{theorem}
\begin{proof}
(1) For $v\in V(G)$ let $B_v=\{v\}\cup V(F^v)$; it is a clique with
$n_2+1\ge2$ vertices, and the sets $B_v$ partition $V(H)$, $H=G\circ K_{n_2}$.
Let $S$ be minimal dominating. A vertex $t\in V(F^v)$ has $N[t]=B_v$, so
$S\cap B_v\ne\varnothing$. If $|S\cap B_v|\ge2$, then $S\cap B_v$ contains a
non-centre $t\in F^v$, and $t$ has no private neighbour ($N[t]=B_v$ contains
another element of $S$) -- contradicting Lemma~\ref{lem:private}. Hence
$|S\cap B_v|=1$ for all $v$. Conversely, any choice of one vertex per block is
dominating, and it is minimal: a chosen $t\in F^v$ is its own private
neighbour, while a chosen $v$ has any vertex of $F^v$ as a private neighbour
($N[t]=B_v$ meets $S$ only in $v$). Thus there are $(n_2+1)^{n_1}$ minimal
dominating sets, and a vertex of $B_v$ lies in those that choose it:
$(n_2+1)^{n_1-1}$ sets. The edge count is $m_1+n_1n_2+n_1\binom{n_2}{2}$.

(2) Let $P_v$ be the $n_2$ pendant vertices at $v$. Since $N[p]=\{p,v\}$ for
$p\in P_v$, we have $S\cap(\{v\}\cup P_v)\ne\varnothing$. If $v\in S$ and some
$p\in S$, then $p$ has no private neighbour; so $v\in S\Rightarrow
S\cap P_v=\varnothing$, and $v\notin S\Rightarrow P_v\subseteq S$. Thus $S$ is
determined by $A=S\cap V(G)$, namely $S_A=A\cup\bigcup_{v\notin A}P_v$, and
every $A\subseteq V(G)$ gives a dominating set. It is minimal: $p\in P_v$ with
$v\notin A$ is its own private neighbour, and $v\in A$ has any $p\in P_v$ as a
private neighbour. Hence there are $2^{n_1}$ minimal dominating sets and each
vertex lies in $2^{n_1-1}$ of them.
\end{proof}

\section{Conclusion}

In this paper we have introduced and studied the \emph{Domination Reciprocal
Mean Square index} \(DRMS\), a new domination-degree-based topological
descriptor. Our contributions can be summarized as follows.

\begin{enumerate}
\item \emph{Bounds.} We established elementary bounds in terms of \(m\),
\(\delta_d\), and \(\Delta_d\), and characterized the equality cases. We then
derived a one-parameter family of lower bounds via the power mean \(M_p(x)\),
which unifies several classical inequalities and produces new intermediate
bounds. Using Cauchy--Schwarz, H\"older, Kantorovich, and chord inequalities,
we obtained sharp bounds in terms of the modified forgotten domination index,
the second domination Zagreb index, the first domination Banhatti index, the
domination Sombor index, and the domination elliptic Sombor index.

\item \emph{Closed forms.} We computed exact expressions for \(DRMS\) on
several standard graph families and on the corona product of two graphs. In the corona cases, the graph is domination-regular, so \(DRMS\) admits a simple expression in terms of the number of edges and the common domination degree; for the standard families, the expressions depend on the domination degrees of the vertices in each part or class.
\end{enumerate}
Several questions remain open. In particular, it would be of interest to
determine the minimum of \(DRMS\) among connected graphs (or trees) of a given
order, to strengthen the relation between \(DRMS\) and the other domination topological indices. We leave these problems for future work.

\end{document}